\documentclass[12pt,a4paper]{amsart} 

\usepackage{amsmath}
\usepackage{amssymb}
\usepackage{amsthm}
\usepackage{graphicx}
\usepackage{amscd}
\usepackage{epic, eepic}
\usepackage{url}
\usepackage{color}
\usepackage[utf8]{inputenc}
\usepackage{comment}
\usepackage{enumerate}
\usepackage[left=1in,top=1in,right=1in,bottom=1in]{geometry}

\theoremstyle{plain}
\numberwithin{equation}{section}

\newtheorem{theorem}{\bf Theorem}[section]
\newtheorem{lemma}[theorem]{\bf Lemma}
\newtheorem{cor}[theorem]{\bf Corollary}

\newtheorem{defi}[theorem]{\bf Definition}

\newtheorem{clm}[theorem]{\bf Claim}

\def\ex{\hbox{\rm ex}}
\newcommand{\eps}{\varepsilon}

\newcommand\cF{{\mathcal F}}
\newcommand\cG{{\mathcal G}}

\title{Hypergraphs without exponents}

\author{Zolt\'an F\"uredi}\thanks{%
\footnotesize{Research supported in part by the Hungarian National Research, Development and Innovation Office NKFIH grants  K116769, KH130371 and by the Simons Foundation Collaboration Grant 317487.
}}
\address{MTA R\'enyi Alfr\'ed Institute for Mathematics,
PO Box 127, H-1364 Budapest, Hungary,
and Department of Mathematics, University of Illinois at Urbana-Champaign,
IL 61801, USA.
}
\email{furedi@renyi.hu}

\author{D\'aniel Gerbner}\thanks{
\footnotesize{Research supported in part by the J\'anos Bolyai Research Fellowship of the Hungarian Academy of Sciences and the National Research, Development and Innovation Office -- NKFIH under the grants K 116769, KH 130371 and SNN 129364.
}}
\address{MTA R\'enyi Alfr\'ed Institute for Mathematics,
PO Box 127, H-1364 Budapest, Hungary.}
\email{gerbner@renyi.hu}

\date{June 14, 2019.} 



\begin{document}

\begin{abstract}
    Here we give a short, concise proof for the following result.
There exists a $k$-uniform hypergraph  $H$  (for $k\geq 5$) without
exponent, i.e., when the Tur\'an function is not polynomial in $n$.
More precisely, we have $\ex(n,H)=o(n^{k-1})$  but it exceeds
$n^{k-1-c}$ for any positive  $c$ for  $n> n_0(k,c)$.

This is an extension (and simplification) of a result of Frankl and
the first author from 1987 where the case $k=5$ was proven.
We conjecture that it is true for $k\in \{3, 4\}$ as well.
\end{abstract}

\maketitle

\section{Notation, the Tur\'an problem}

We start with some standard notation.
A $k$-{\em graph} (or $k$-{\em uniform hypergraph}) $H$ is a pair $(V, E)$ with $V=V(G)$ a set of vertices, and $E=E(G)$ a collection of $k$-sets from $V$, which are the hyperedges (or $k$-edges) of $H$.
The $s$-\emph{shadow}, $\partial_s H$, is the family of $s$-sets contained in the hyperedges of $H$. So $\partial_1 H$ is the set of non-isolated vertices, and $\partial_2 H$ is a graph.
We write $[n]$ for $\{1,2, . . . n\}$.
Given a set $A$ and an integer $k$, we write $A\choose k$ for the set of $k$-sets of $A$.
When there is no confusion, we may also use ‘edge’ for ‘$k$-edge’.
The {\em complete} $k$-graph on $n$ vertices is the $k$-graph $K_n^{(k)}= ([n],{[n]\choose k})$.
Let $I_k(i)$ denote the $k$-uniform hypergraph consisting of two hyperedges sharing exactly $i$ vertices.
The $k$-graph $H$ is $k$-{\em partite} if there exists a partition $\{ P_1, \dots, P_k\}$ of $V(H)$ such that for every edge $e\in E(H)$ and part $P_i$ we have $|e\cap P_i|=1$.
The {\em complete} $k$-partite $k$-graph $K_k(P_1, \dots, P_k)$ has all of such edges, $|E(K_k(P_1, \dots, P_k))|= |P_1|\times \dots \times |P_k|$.

Given a family of $k$-graphs $\cF$, we say that a $k$-graph $H$ is $\cF$-{\em free} if it contains no member of $\cF$ as a subgraph. We write $\ex(n,\cF)$ (or $\ex_k(n, \cF)$ if we want to emphasize $k$) for the maximum number of $k$-edges that can be present in an $n$-vertex $\cF$-free $k$-graph. The 
function $\ex(n,\cF)$ is referred to as the {\em Tur\'an number} of $\cF$. We leave out parentheses whenever it is possible, e.g., in case of $|\cF|=1$ we write $\ex(n, F)$ instead of $\ex(n, \{ F\})$.

\section{Rational exponents and non-polynomial Tur\'an functions}

Erd\H os and Simonovits (see~\cite{erdos,erdsim}) {\em conjectured} that for any rational $1\le\alpha\le 2$ there exists a graph $F$ with
\begin{equation}\label{eq21}
  \ex_2(n,F)=\Theta(n^\alpha)
\end{equation} and conversely, for every graph $F$ we have
\begin{equation}\label{eq22}\ex_2(n,F)=\Theta(n^\alpha)
   \end{equation}
for some rational $\alpha$. Bukh and Conlon~\cite{bc} showed that the first conjecture holds if we can forbid finite families of graphs. For a single graph, it is still unknown.

For hypergraphs 
 Frankl~\cite{frankl} showed that all rationals occur as exponents of $\ex_k(n,\cF)$ for some $k$ and for some finite family $\cF$ of $k$-uniform hypergrahs. Fitch~\cite{fitch} 
 showed that for a fixed $k$ all rational numbers between $1$ and $k$ occur as exponents of $\ex_k(n,\cF)$ for some family $\cF$ of $k$-uniform hypergraphs.

We say that a function $f(n): {\mathbb N}\to {\mathbb R}$ has {\em no exponent} if there is no real $\alpha$  such that $f(n)= \Theta(n^\alpha)$.
In other words, the order of magintude of $f(n)$ is not a polynomial.

Brown, Erd\H os, and S\'os \cite{bes} proposed the following problem.
Determine (or estimate) $f_k(n, v, e)$, i.e.,  the maximum number of edges in a $k$-uniform, $n$-vertex hypergraph in which no $v$ vertices span $e$ or more edges. This is a Tur\'an type problem: Let $\cG_k(v, e)$ be the family of $k$-graphs, each member having $e$ edges and at most $v$ vertices, then $f_k(n,v,e)= \ex_k(n, \cG_k(v,e))$.

Ruzsa and Szemer\'edi~\cite{rsz} showed that if a 3-uniform hypergraph does not contain three hyperedges on six vertices, then it has $o(n^2)$ edges, and they also gave a construction with $n^{2-o(1)}$ hyperedges.
The assumption on the hypergraph is equivalent to forbidding the following two sub-hypergraphs $\{123,124\}$ (a pair covered twice) and $\{123,345,561\}$ (a linear triangle). They proved
\begin{multline}\label{eq23}
n^{2-o(1)}< \frac{1}{10}n r_3(n) < f_3(n,6,3)-(n/2) \\ \leq \ex_3(n, \left\{ \{123,124\}, \{123,345,561\}\right\} )
   \leq f_3(n,6,3) = o(n^2).
\end{multline}
(For the definition of $r_3(n)$, see the paragraph containing~\eqref{eq53} in Section~\ref{sec:tools}).
Thus they found a family of two hypergraphs such that not only its Tur\'an number does not have a rational exponent, it does not have an exponent at all.
This is the famous $(6,3)$-theorem, $f_3(n,6,3)$ is non-polynomial.

Erd\H os, Frankl, and R\"odl~\cite{efr} extended this to every $k$ proving
 $f_k(n, 3k-3, 3)= o(n^2)$ but $\lim _{n\to \infty} f_k(n, 3k-3, 3)/n^{2-\eps} = \infty$ for all $\eps>0$ ($k\geq 3$ and $\eps$ are fixed, $n \to \infty)$.
The proof of the upper bound here and in~\eqref{eq23} are based only on Szemer\'edi's regularity lemma~\cite{szem2}.

\section{Single hypergraphs with no exponents}

Answering a question of Erd\H os, a single 5-uniform hypergraph with no exponent was presented in~\cite{ff87}:

\begin{theorem}[Frankl and F\"uredi~\cite{ff87}]\label{frfu}
Let $H=\{12346,12457,12358\}$. Then $\ex_5(n,H)=o(n^4)$ but $\ex_5(n,H)\neq O(n^{4-\varepsilon})$ for any $\varepsilon>0$.
\end{theorem}

One aim of this paper is to give a short proof for this result.
The original proof heavily relied on the delta-system method, we can get rid of that.
We also extend it for all $k\geq 5$.
We {\em conjecture} that examples with no exponents should exist for $k=3$ and $4$, too.


\begin{defi}\label{qkr}
Let us consider three disjoint sets of vertices $A=\{a_1,\dots,a_{k-r}\}$, $B=\{b_1,\dots, b_r\}$ and $C=\{c_1,\dots, c_r\}$. Let $Q_k(r)$ denote the $k$-uniform hypergraph consisting of all the hyperedges of the form $A\cup \left(B\setminus \{b_i\}\right) \cup \{c_i\}$, for $1\le i \le r$.
\end{defi}

So $|E(Q_k(r)|=r$ and $|V(Q_k(r)|=k+r$.
To avoid trivialities we suppose that $r\geq 2$ since $Q_k(0)$ is an empty hypergraph and $Q_k(1)$ has only one hyperedge.
In this paper we study $\ex_k(n,Q_k(r))$ for every pair of values $k$ and $r$, $k\geq r\geq 2$, and we either determine the order of magnitude or show that there is no exponent.

Note that $Q_k(2)=I_k(k-2)$ (two $k$-edges meeting in $k-2$ elements). The study of the Tur\'an number of $I_k(i)$ has been initiated by Erd\H os~\cite{erdos}. Frankl and F\"uredi~\cite{ff85} proved that $\ex_k(n, I_k(i))= \Theta(n^{\max \{i, k-i-1\}})$ for $0\leq i\leq k-1$.
This gives $\ex_k(n,Q_k(2))=\Theta(n^{k-2})$ for $k\geq 3$ and $\ex_2(n, Q_2(2))= \Theta(n)$.

Our main result is the following theorem.

\begin{theorem}\label{main} If $k\geq r\ge 3$ and $r\geq (k/2)+1$, then $\ex_k(n,Q_k(r))=\Theta(n^{k-1})$.

If $k\geq r\ge 3$ and $r\leq (k+1)/2$, then $\ex_k(n,Q_k(r))=o(n^{k-1})$ but $\ex_k(n,Q_k(r))\neq O(n^{k-1-\varepsilon})$ for any $\varepsilon>0$.
\end{theorem}

Note that $Q_5(3)=\{12346,12457,12358\}$, so this Theorem is indeed an extension of Theorem~\ref{frfu}.
Since $Q_k(3)\subset \dots \subset Q_k(k)$, we obviously have
\[
\ex_k(n,Q_k(3) ) \leq \ex_k(n,Q_k(4) )\leq \dots \leq \ex_k(n,Q_k(k) ).
\]
So to prove Theorem~\ref{main} we need to show that for $k\geq r\geq 3$ as $n\to \infty$ we have  \newline${}$\quad\quad(\ref{main}.a)\quad $\ex_k(n,Q_k(k))= O(n^{k-1})$,
\newline${}$\quad\quad(\ref{main}.b)\quad $\ex_k(n,Q_k(r))= \Omega(n^{k-1})$ if $k\leq 2r-2$,
\newline${}$\quad\quad(\ref{main}.c)\quad  $\ex_k(n,Q_k(r))= o(n^{k-1})$ if $k\geq 2r-1$,
\newline${}$\quad\quad(\ref{main}.d)\quad $\ex_k(n,Q_k(3))= \Omega(n^{k-1-\eps})$ if $k\geq 5$, $\forall \eps>0$ fixed. \newline
We emphasize that to prove that $Q_k(3)$ has no exponent (for $k\geq 5$), we only use
the hypergraph removal lemma (Lemma \ref{le51}) and our lower bound construction from Section~\ref{sec:constr}.

\medskip
\noindent
{\em Problem.}\enskip {\em Determine} $\limsup_{n\to \infty} \ex_k(n,Q_k(r))/n^{k-1}$ {\em for} $4\leq k\leq 2r-2$.
\medskip

The rest of the paper is organized as follows.
In Section~\ref{sec:principal} we discuss a strongly related problem, in  Sections~\ref{sec:tools} and~\ref{sec:FR} the necessary tools are presented,  Section~\ref{sec:upper} contains the proof of the upper bounds (\ref{main}.a) and (\ref{main}.c),  Section~\ref{sec:constr1} is a simple construction to establish the lower bound (\ref{main}.b), and
 our most interesting construction for the lower bound (\ref{main}.d) is presented in Section~\ref{sec:constr}.
Finally, a simple proof for (\ref{main}.d) is presented in Section~\ref{sec:10} for the special case $k=2r-1$.

\section{Principal families}\label{sec:principal}

An easy averaging argument shows that $\ex(n,\cF)/{n \choose k}$is nonincreasing and hence tends to a limit as $n\to \infty$.  This limit, denoted by $\pi(\cF)$, is the {\em Tur\'an density} of $\cF$.
The Tur\'an (density) problem for $k$-graphs is this: given a family $\cF$, determine $\pi(\cF)$.
This question for $2$-graphs, i,e., for a family of ordinary graphs $\cG$,  has been completely answered by the Erd\H os-Stone-Simonovits Theorem, which states $\pi(\cG)= (m -2)/(m-1)$, where $m$ is the smallest chromatic number of graphs in $\cG$.
Hence
\begin{equation} \label{eq41}
\pi(\cG)= \min_{G\in \cG}\pi(G).
  \end{equation}
Thus Tur\'an density is {\em principal} among ordinary graphs.

  By contrast very few Tur\'an densities of $k$-graphs are known (although Pikhurko~\cite{pik} gave infinitely many values).
Nonprincipality for $3$-graphs was conjectured by
Mubayi and R\"odl~\cite{mr}, and first exhibited by Balogh~\cite{ba}.
Mubayi and Pikhurko~\cite{mp} gave the first example of a {\em nonprincipal pair} of $3$-graphs, i.e.
a pair $F, F'$ with $\pi(F, F') < \min\{ \pi (F), \pi (F')\} $.
The simplest pair is due to Falgas-Ravry and Vaughan~\cite{ZF_VFR}, who proved $\pi(K_4^-, F_{3,2})=5/18=0.2777\dots$,
 where $E(K_4^-)= \{123, 124, 134\}$ and $E(F_{3,2})=\{123, 124, 125, 345\}$. On the other hand
there is a lower bound $\pi(K_4^-)\geq 2/7$ from~\cite{ff84}, and in~\cite{fps} it was proved that $\pi(F_{3,2})=4/9$.

Equation~\eqref{eq41} implies that, in case of ordinary graphs, if  $\min_{G\in \cG} \chi(G)>2$ then always exists a $G\in \cG$ such that
\begin{equation} \label{eq42}
  \ex(n, \cG)= (1+o(1)) \ex(n, G).
  \end{equation}
When bipartite graphs are involved then such a strong principality does not hold.
Erd\H os and Simonovits~\cite{erdsim} proved that
 $\ex(n, \{C_4, C_5\})=(1+o(1))(n/2)^{3/2}$, on the other hand we have
 $\ex(n, C_4)=(\frac{1}{2}+o(1))n^{3/2}$ and $\ex(n, C_5)= \lfloor n^2/4\rfloor$ (for $n\geq 6$).
So, instead of~\eqref{eq42}, Erd\H os and Simonovits~\cite{erdsim} made the following {\em compactness conjecture}
(in fact, we can call it {\em weak} principality), that any finite family $\cG$ of graphs (with $\ex(n,\cG)\neq O(1)$)  contains a single graph $G$ such that
\begin{equation} \label{eq43}
\ex_2(n,\cG)=\Theta(\ex_2(n,G)).
 \end{equation}
This conjecture with the result of Bukh and Conlon (mentioned after~\eqref{eq22}) would imply conjecture~\eqref{eq21}.

The upper bound in the Ruzsa-Szemer\'edi $(6,3)$-theorem (i.e., $\ex_3(n, \{ I_3(2), Q_3(3) \})=o(n^2)$, see~\eqref{eq23}) shows that there is no compactness for hypergraphs.
Indeed, the Tur\'an number of $I_3(2)$ is $n(n-1)/6+O(n)$  (Steiner triple systems are extremal) and $\ex(n, Q_3(3))\geq \binom{n-1}{2}$ (because the centered family $\{ f: f\in \binom{[n]}{3}, 1\in f \}$ 
 does not contain linear triangles). Actually, it is known~\cite{ff87} that  $\ex(n, Q_3(3))= \binom{n-1}{2}$ for $n> n_0$, so both of these hypergraphs have quadratic Tur\'an numbers.

\section{Lemmas and tools}\label{sec:tools}

The following observation, due to Erd\H os and Kleitman, is one of the basic tools to determine the order of magnitude of the size of a $k$-graph $H$: Every $k$-graph $H$ has a
 $k$-partition of its vertices $V(H)= P_1\cup \dots \cup P_k$ into almost equal parts $\left(\bigl\lvert|P_i|-|P_j|\bigl\rvert\leq 1\right)$ such that
 for the $k$-partite subhypergraph $H'$ with $E(H'):= E(H) \cap E(K_k(P_1, \dots ,P_k))$, one has
 \begin{equation}\label{eq51}
  \frac{k!}{k^k}|E(H)|\leq  |E(H')|\leq |E(H)|.
 \end{equation}

Suppose $n\geq r\geq t\geq 1$ are integers. An $r$-graph $H$ on $n$ vertices is called an $(n,r,t)$-{\em packing} if $|e\cap e'|<t$ holds for every $e,e'\in E(H)$, $e\neq e'$. The maximum of $|E(H)|$ of such packings is denoted by $P(n,r,t)$.
Since then $\binom{n}{t}\geq |\partial_t H|=\binom{r}{t}|E(H)|$, we have $P(n,k,t)\leq \binom{n}{t}/\binom{r}{t}$.
It is known that $P(n,r,t)=(1+o(1))\binom{n}{t}/\binom{r}{t}$ when $r$ and $t$ are fixed and $n$ tends to infinity.
(Even {\em perfect} packings, i.e., Steiner systems $S(n,r,t)$'s, exist if some divisibility constraints hold and $n$ is sufficiently large.) We only use the following easy statement: If $r$ is fixed and $n\to\infty$ then
 \begin{equation}\label{eq52}
   P(n,r,t)\geq \binom{n}{t}/ \binom{r}{t}^2 = \Omega(n^t).
 \end{equation}

Let $k$ and $n$ be positive integers.
A set of numbers $A$ is called ${\rm AP}_k$-{\em free} if it does not contain $k$  distinct elements forming an arithmetic progression of length $k$. As usual, let $r_k(n)$ denote the maximum size of an ${\rm AP}_k$-free sequence $A\subseteq [n]$.
The celebrated Szemer\'edi's theorem~\cite{szem} states that for a fixed $k$ as $n\to \infty$ we have
 \begin{equation}\label{eq53}
    r_k(n)=o(n).
 \end{equation}
(The case $r_3(n)=o(n)$ was proved much earlier by K. F. Roth).

Let $k$ be an integer and $p$ be a prime, $p>k$. We say that $S\subseteq \{0,\dots,p-1\}$ is $k$-\emph{good} if for any $m_1,m_2,m_3\in \{-k,-k+1,\dots,-1\}\cup\{1,\dots,k\}$ and $s_1,s_2,s_3\in S$
the following equations hold:
\begin{equation*}
\left.\begin{array}{l}
    m_1+m_2+m_3=0 \,\,\,\,\,\text{and}\\[4pt]
    m_1s_1+m_2s_2+m_3s_3=0\end{array}\right\}
  \quad {\rm imply }\quad s_1=s_2=s_3.
\end{equation*}
Here addition and multiplication are taken modulo $p$. Let $s_k(p)$ denote the size of the largest $k$-good set. The following result is an easy extension of Behrend's construction, see, e.g., Ruzsa~\cite{ruzsa1, ruzsa2}:
There is a $c_k>0$ such that
\[  p \exp[ - c_k \sqrt{\log p}] < s_k(p).
\]
We only need that if $k$ and $\eps>0$ are fixed and $p\rightarrow\infty$, then
\begin{equation}\label{eq55}
  s_k(p)> p^{1-\eps}
\end{equation}
Note that a $k$-good set cannot contain a (strictly increasing) arithmetic progression of length $3$, so $s_k(p)\leq r_3(p)$ and $r_3(p)=o(p)$ by Roth's theorem, see~\eqref{eq53}.

We will also use the so-called hypergraph removal lemma. It (together with other versions of hypergraph regularity) was developed by several groups of researchers, see~\cite{gowers,nrs,rs,rs2,tao}.

\begin{lemma}[Hypergraph Removal Lemma]\label{le51} For any $\varepsilon>0$ and integers $\ell\ge k$, there exist $\delta>0$ and an integer $n_0$ such that the following statement holds. Suppose $F$ is a $k$-uniform hypergraph on $\ell$ vertices and $H$ is a $k$-uniform hypergraph on $n\ge n_0$ vertices, such that $H$ contains at most $\delta\binom{n}{\ell}$ copies of $F$. Then one can delete at most $\varepsilon\binom{n}{k}$ hyperedges from $H$ such that the resulting hypergraph is $F$-free.

\end{lemma}

\section{Szemer\'edi's $r_k(n)=o(n)$ by Frankl and R\"odl}\label{sec:FR}

Recall that $I_k(i)$ denotes the $k$-uniform hypergraph consisting of two hyperedges sharing exactly $i$ vertices.
Frankl and R\"odl~\cite{fr} generalized the lower bound of the celebrated $(6,3)$-theorem (i.e.,~\eqref{eq23}) of Ruzsa and Szemer\'edi~\cite{rsz} as follows.

\begin{theorem}[\cite{fr}]\label{th:fr} For 
any integer $k\geq 3$  
there exists a $c_k'>0$ such that for all $n\geq k$
\begin{equation*} 
    c_k'\times r_k(n)\times  n^{k-2}\leq \ex_k(n,\{Q_k(k),I_k(k-1)\}). 
\end{equation*}
\end{theorem}

They conjectured $\ex_k(n,\{Q_k(k),I_k(k-1)\})=o(n^{k-1})$ and proved the case $k=4$ (the case $k=3$ is part of~\eqref{eq23}).
In order to prove $\ex_4(n,\{Q_4(4),I_4(3)\})=o(n^3)$ they developed a hypergraph removal lemma for the 3-uniform case.
They also described how the hypergraph removal lemma (Lemma~\ref{le51}) would imply the general upper bound $o(n^{k-1})$.
Since then Lemma~\ref{le51} has been proved, we have the following result.

\begin{cor}\label{co:fr2} For any $k\geq 2$ we have $\ex_k(n,\{Q_k(k),I_k(k-1)\})=o(n^{k-1})$.
\end{cor}

Note that Theorem~\ref{th:fr} and Corollary~\ref{co:fr2} imply Szemer\'edi's theorem: $r_k(n)=o(n)$.

The upper bound in Corollary~\ref{co:fr2} supplies a non-compact pair for $k$-graphs.
The Tur\'an number of $I_k(k-1)$ is $\Theta(n^{k-1})$ (see \eqref{eq52}) and $\ex(n, Q_k(k))={n\choose k-1} +O(n^{k-2})$ (see~\cite{ff87} and~\cite{206}).

Since the above corollary plays such an important role in our main result, we include its few line proof for the sake of completeness.

\bigskip
\noindent
{\em Proof} of Corollary~\ref{co:fr2}. \enskip
Let $H$ be a $Q_k(k)$ and $I_k(k-1)$-free $k$-graph on $n$ vertices.
We will give an upper bound on its size.
By~\eqref{eq51} we may suppose that $H$ is $k$-partite with parts $P_1,\dots,P_k$.
Consider its shadow $\partial H$, it is a $(k-1)$-uniform hypergraph.
Since $H$ is $I_k(k-1)$-free, each $f\in \partial H$ is contained in a unique $e(f)\in E(H)$.
We get $\binom{k}{k-1}|E(H)|=|E(\partial H)|$. This already gives $|E(H)|=O(n^{k-1})$.

Every edge $e\in E(H)$ induces a complete subhypergraph $K_k^{(k-1)}$ in $\partial H$.
We claim that these are the only cliques of size $k$ in $\partial H$.
Consider $K$ a copy of $K_k^{(k-1)}$ in $\partial H$. 
Then $|P_i\cap V(K)|=1$ for each $P_i$.
If $e(f)=V(K)$ for some $f\in E(K)$ then $K$ is the clique generated by $V(K)=e(f)\in E(H)$.
Otherwise, when $e(f)\neq V(K)$ for each $f\in E(K)$,
the $k$ hyperedges $\{ e(f): f\in E(K)\}$ form a copy of $Q_k(k)$. This contradiction implies that
 $\partial H$ is indeed the edge-disjoint union of cliques induced by the edges of $H$, and these are the only $k$-cliques in $\partial H$.

Therefore, the number of copies of $K_k^{(k-1)}$ in $\partial H$ is $O(n^{k-1})=o(n^{|V(K)|})$. Then by the hypergraph removal lemma (Lemma~\ref{le51}) there exists a subhypergraph $H'$, $E(H')\subset E(\partial H)$, so that $E(H')$  meets  every copy of $K_k^{(k-1)}$ in $\partial H$ and $|E(H')|=o(n^{k-1})$.
For such an $H'$ we have $|E(H)|\leq |E(H')|$, finishing the proof.
\qed

\section{Proof of Theorem \ref{main}, upper bounds}\label{sec:upper}

In this section we prove (\ref{main}.a) and (\ref{main}.c), the upper bounds for $\ex_k(n,Q_k(r))$.

Let $H$ be a $Q_k(r)$-free $k$-graph on $n$ vertices.
We will give an upper bound on $|E(H)|$.
By~\eqref{eq51} we may suppose that $H$ is $k$-partite with parts $P_1,\dots,P_k$.
For a hyperedge $e\in E(H)$, let $D(e)\subseteq [k]$ denote the set of integers $i$ such that there is another hyperedge $e'\in E(H)$ that differs from $e$ only in $P_i$, $e\setminus P_i= e'\setminus P_i$.
Note that $|D(e)|<r$ because $H$ is $Q_k(r)$-free.

By the pigeonhole principle there is a set $D\subset \{1,\dots,k\}$ such that there are at least $|E(H)|/2^k$ hyperedges $e\in E(H)$ with $D(e)=D$.
Let $H'$ be the $k$-graph of these edges, $E(H'):= \{ e\in E(H): D(e)=D\}$.
Set $\ell:= k-|D|$, we have $\ell\geq k-r+1$, $\ell\geq 1$.

Let $T$ be an edge of the complete $|D|$-partite hypergraph with parts $\{ P_i: i\in D\}$, i.e.,
  $|T|=|D|$ and $|T\cap P_i|=1$ for each $i\in D$.
($D$ might be the empty set).
There are at most $O(n^{k-\ell})$ appropriate $T$.
Define $H'[T]$ as the {\em link} of $T$ in $H'$, i.e.,
  it is an $\ell$-graph with edges $\{ e\setminus T: T\subset e\in E(H')\}$.

Observe first that $H'[T]$ is $I_\ell(\ell-1)$-free. Indeed, two hyperedges of $H'[T]$ sharing $\ell-1$ vertices would mean two hyperedges in $H'$ sharing $k-1$ vertices such that their only difference is in a part not belonging to $D$.
So every $(\ell-1)$-element set is contained in at most one hyperedge in $H'[T]$, thus $|H'[T]|\le \binom{n}{\ell-1}$.
Since $|E(H')|= \sum_T |E(H'[T])|$, we obtained
\begin{equation}\label{eq71}
|E(H)|=O(|E(H')|)=O(n^{k-l})\binom{n}{\ell-1} =O(n^{k-1}),
\end{equation}
completing the proof of (\ref{main}.a).

Finally, let us assume $k\ge 2r-1$, i.e., $\ell\geq r$.
We claim that in this case $H'[T]$ is also $Q_\ell(\ell)$-free.
Indeed, if we add $T$ to the hyperedges of a copy of $Q_\ell(\ell)$ from $H'[T]$, we obtain a $Q_k(\ell)$ in $H'$.
Since $Q_k(\ell)$ contains a $Q_k(r)$, this is a contradiction.
Thus we have $|E(H'[T])|=o(n^{\ell-1})$ by Corollary~\ref{co:fr2}.
We complete the proof as in~\eqref{eq71}
\[
|E(H)|=O(|E(H')|)=O(n^{k-l})\times o(n^{\ell-1}) =o(n^{k-1}).
  \qed
\]

For the case $k\ge 2r-1$ we actually proved that $\ex_k(n,Q_k(r))\le \ex_k(n,\{Q_k(k),I_k(k-1)\})$. This is $o(n^{k-1})$ by Corollary \ref{co:fr2}. Theorem \ref{th:fr} shows that this way the upper bound cannot be improved significantly, because  $\ex_k(n,\{Q_k(k),I_k(k-1)\})= \Omega( r_k(n)\times  n^{k-2})$. In Section \ref{sec:constr} we will present the slightly weaker lower bound $\Omega( s_k(n)\times  n^{k-2})$ for $\ex_k(n,Q_k(r))$.

\section{Proof of Theorem \ref{main}, the polynomial range}\label{sec:constr1}

In this section we prove the lower bound (\ref{main}.b) by giving a construction.

Since $k\le 2r-2$, we have $r-1\geq k+1-r \geq 1$.
Let $X$ and $Y$ be two disjoint sets,  $|X|=\lfloor n/2\rfloor$ and $|Y|=\lceil n/2\rceil$.
Let $H^1$ be an $(|X|, r-1, r-2)$-packing of maximum size, i.e., an $(r-1)$-uniform hypergraph such that any two hyperedges share at most $r-3$ vertices. By~\eqref{eq52} we have $|E(H^1)|= \Theta(n^{r-2})$.
Let $H^2$ be the complete $(k-r+1)$-uniform hypergraph with vertex set $Y$.
Finally, let $H^3$ be the $k$-graph with vertex set $X\cup Y$ having as hyperedges all the $k$-sets that are unions of a hyperedge of $H^1$ and a hyperedge of $H^2$. Then $H^3$ has $\Theta(n^{k-1})$ hyperedges.
We claim that $H^3$ is $Q_k(r)$-free.

Assume, on the contrary, that there is a copy of $Q_k(r)$ in $H^3$, $E(Q_k(r))=\{ f_1, f_2, \dots, f_r\}$.
Note that $|\cap f_i|=k-r < (k-r+1)\leq r-1$ and the symmetric differences $\{ f_i\bigtriangleup f_j : 1\leq i<j\leq r\}$ are all distinct $4$-element sets.
Consider, first, the case when for some $i\neq j$ we have $f_i\cap X=f_j\cap X$. Then all $f_t\cap X$ are identical.
Indeed, if there exists an $f_t\cap X\neq f_i\cap X$, then these two $(r-1)$-sets have symmetric difference at least 4, so it should by exactly 4, and then $(f_i\cap X)\bigtriangleup (f_t\cap X)$ and  $(f_j\cap X)\bigtriangleup (f_t\cap X)$ are identical 4-element sets, a contradiction. Then $|\cap f_i|\geq r-1$, a contradiction.

From now on, we may suppose that the $(r-1)$-element sets $\{ f_i\cap X\}$ are all distinct.
Then, because  $|(f_i\cap X)\bigtriangleup (f_j\cap X)|\geq 4$ we have that $f_i\cap Y=f_j\cap Y$ for all $1\leq i<j\leq r$.
Hence $|\cap f_i|\geq k-r+1$, a final contradiction. \qed


\section{Proof of Theorem \ref{main}, a non-polynomial lower bound}\label{sec:constr}

In this section we prove the lower bound (\ref{main}.d) by giving a construction.
We will show that if $n=kp$, where $k\geq 5$ and $p$ is a prime, then $\ex(n,Q_k(3))\geq p^{k-2}s_k(p)$. As  $\ex(n,Q_k(3))$ is monotone in $n$ and there is a prime between $n/2k$ and $n/k$, this and~\eqref{eq55} give the desired bound $\Omega(n^{k-1-o(1)})$ for $\ex(n,Q_k(3))$.

Let the vertex set $V$ consist of the pairs $(i,j)$ with $1\le i\le k$ and $0\le j\le p-1$. Choose two integers  $0\leq \alpha,\beta \le p-1$ and a $k$-good set $S\subset \{0,\dots,p-1\}$ of size $s_k(p)$.  Suppose that $m_1,\dots, m_k\in\{1,\dots,k\}$  are distinct integers (i.e., a permutation of $[k]$).
We define a $k$-partite $k$-graph $F=F(S,\alpha,\beta)$ on $V$ with parts $P_i:=\{ (i,j): 0\leq j\leq p-1\}$.  A $k$-set $\{(1,x_1),(2,x_2),\dots, (k,x_k)\}$ is a hyperedge of $F$ if the following two equations hold.
\begin{align*}
\left(\sum_{i=1}^k x_i\right)&=\alpha   \pmod{p}, \\
\left(\sum_{i=1}^km_ix_i\right) &\in S+\beta   \pmod{p}. \end{align*}

We have $|F(S,\alpha,\beta)|=p^{k-2}s_k(p)$. Indeed, for any $s\in S$ we can pick $k-2$ values $x_3, \dots, x_k$ arbitrarily, and since $m_1\neq m_2$, the above two equations uniquely determine $x_1$ and $x_2$.

\begin{clm} $F$ is $Q_k(3)$-free.
\end{clm}

\begin{proof}[Proof of Claim] Suppose, on the contrary, that there is  a copy of $Q_k(3)$ in $F$, and let $A,B,C$ be the sets of vertices as in Definition~\ref{qkr}. Without loss of generality we may assume that
$A= \{ (i,x_i): 4\leq i \leq k \}$, $b_i=(i,x_i)$ ($i=1,2,3$), and $c_i=(i,y_i)$ ($i=1,2,3$). 
Then the constraints in the definition of $F$ imply the following equations.
\[\left(\sum_{i=1}^kx_i\right)+y_1-x_1=\alpha  \pmod{p}\]

\[\left(\sum_{i=1}^kx_i\right)+y_2-x_2=\alpha  \pmod{p}\]

\[\left(\sum_{i=1}^kx_i\right)+y_3-x_3=\alpha  \pmod{p}\]

\[\left(\sum_{i=1}^km_ix_i\right)+m_1(y_1-x_1)=s_1+\beta \pmod{p} \]

\[\left(\sum_{i=1}^km_ix_i\right)+m_2(y_2-x_2)=s_2+\beta \pmod{p} \]

\[\left(\sum_{i=1}^km_ix_i\right)+m_3(y_3-x_3)=s_3+\beta \pmod{p} \]%
for some $s_1s_2,s_3\in S$. Define $u$ and $v$ as $u:=\alpha-(\sum_{i=1}^k x_i)$ and $v:=(\sum_{i=1}^km_ix_i)-\beta$.
We obtain
\begin{equation}\label{eq91}
   y_j-x_j=u, \pmod{p} \quad \text{ for } j=1,2,3
\end{equation}
and
\begin{equation}\label{eq92}
   v+m_ju=s_j \pmod{p} \quad \text{ for } j=1,2,3.\end{equation}%
These imply
\[  (v+m_1u -s_1) (m_2-m_3)+ (v+m_2u -s_2) (m_3-m_1)+(v+m_3u -s_3) (m_1-m_2)=0.
\]
Rearranging
\[(m_3-m_2)s_1+(m_1-m_3)s_2+(m_2-m_1)s_3=0 \pmod{p}.\]
As $S$ is a $k$-good set and $1\leq |m_i-m_j|\le k$, we have $s_1=s_2=s_3$.
Then~\eqref{eq92} gives $v+m_1u=v+m_2u=v+m_3u$ implying $u=0$. Then~\eqref{eq91}
 gives $x_j=y_j$ (for $j=1,2,3$), a contradiction.
\end{proof}

\section{A lower bound for the case $k=2r-1$}\label{sec:10}

In this section we present a simple construction implying the lower bound in (\ref{main}.d) for the case $k=2r-1$. 
It gives $\ex(n,Q_{2r-1}(r))\geq \Omega(r_r(n)n^{k-2})$, a stronger lower bound than the one in the previous section. 
The construction is similar to the one in Section~\ref{sec:constr1}. 

Start with an $r$-graph $H^1$ with a set $V_1$ of $\lfloor n/2\rfloor$ vertices and $\Omega(r_r(n)n^{r-2})$ hyperedges that is both $Q_r(r)$-free and $I_r(r-1)$-free. The existence of such hypergraphs was proved by Frankl and R\"odl~\cite{fr}, see Theorem \ref{th:fr}. Add a set $V_2$ of $\lceil n/2\rceil$ new vertices and take all $k$-edges containing an $r$-edge of $H^1$ and $r-1$ vertices from $V_2$. This hypergraph $H$ has $\Omega(r_r(n)n^{k-2})$ hyperedges.

We claim that $H$ is $Q_k(r)$-free ($k=2r-1$). 
Suppose, on the contrary, that $H$ contains a copy of $Q_k(r)$, and let $A$, $B=\{ b_1, \dots, b_r\}$, and $C=\{c_1, \dots, c_r \}$ be the sets of vertices as in Definition \ref{qkr}.
Since $|B|, |C|>r-1$ they both share at least one element with $V_1$, say $b_i\in B\cap V_1$ and $c_j\in V_1\cap C$. 
If $c_i$ is not in $V_1$, then $e_i:=A\cup B\setminus\{b_i\}\cup \{c_i\}$ has less elements in $V_1$ than $e_j:=A\cup B\setminus\{b_j\}\cup \{c_j\}$ does. It is a contradiction as both $e_i\cap V_1$ and $e_j\cap V_1$ are hyperedges of $H^1$. 
We obtained that $b_i\in V_1\cap B$ implies $c_i\in V_1\cap C$.

If there exists a $b_t\in V_2\cap B$ then $c_t$ also must belong to $V_2$. Otherwise, $|e_t\cap V_1|> |e_i\cap V_1|$, a contradiction. 
In this case $b_i, c_i \in V_1$ and $b_t, c_t \in V_2$ imply that 
$e_t:=A\cup B\setminus\{b_t\}\cup \{c_t\}$ shares $r-1$ elements with $e_i=A\cup B\setminus\{b_i\}\cup \{c_i\}$ inside $V_1$, which contradicts the $I_r(r-1)$-free property of $H^1$.

Hence we may assume that each $b_t\in B$ belongs to $V_1$. Then $C\subset V_1$, too, so $A\subset V_2$. 
Then the $r$-edges $B\setminus\{b_t\}\cup \{c_t\}$ form a copy of $Q_r(r)$ in $H^1$. This final contradiction completes the proof.


\end{document}